\documentclass[10pt,a4paper,draft]{article}

\usepackage{geometry}
\usepackage[T1]{fontenc}
\usepackage[utf8]{inputenc}
\usepackage{authblk}
\usepackage{amsmath}
\usepackage{amssymb}
\usepackage{amsfonts}
\usepackage{amsthm}
\usepackage{enumerate}
\usepackage{paralist}
\usepackage{amsfonts}
\usepackage{tensor} 
\usepackage{lmodern}
\usepackage{upgreek}
\usepackage{mathrsfs}
\usepackage{nicefrac}
\usepackage{verbatim}
\usepackage{amsmath}
\usepackage[all, cmtip]{xy}\usepackage{amssymb}
\usepackage{amsthm}
\usepackage{extpfeil}
\usepackage{bbm}
\usepackage{paralist}
\usepackage[colorlinks,linkcolor=blue,citecolor=gray,urlcolor=blue,pdftex]{hyperref}

\usepackage{enumitem}

\usepackage{titlesec}
\titleformat*{\section}{\large\bfseries}
\titleformat*{\subsection}{\large\sffamily}

\usepackage[titletoc,toc,title]{appendix}

\usepackage[all]{xy}
\usepackage{tikz}
\usepackage{relsize}
\usepackage{comment}
\usepackage{mathtools}
\usetikzlibrary{decorations.pathmorphing}

\usepackage{graphicx}

\usepackage{float}

\theoremstyle{plain}
\newtheorem{thm}{Theorem}

\newtheorem{prop}[thm]{Proposition}
\newtheorem{lem}[thm]{Lemma}
\newtheorem{cor}[thm]{Corollary}

\newtheorem*{prb*}{Problem}

\theoremstyle{definition}

\newtheorem{ex}[thm]{Example}

\numberwithin{equation}{section}

\newcommand{\cm}[1]{}

\newcommand\wt[1]{\widetilde{#1}}

\renewcommand{\k}{\mathbbm{k}}

\newcommand\K{\mathcal{K}}

\makeatletter
\def\dual#1{\expandafter\dual@aux#1\@nil}
\def\dual@aux#1/#2\@nil{\begin{tabular}{@{}c@{}}#1\\#2\end{tabular}}
\makeatother

\newcommand{\Xm}{X^{\langle m\rangle}}
\newcommand{\Km}{\mathcal{K}^{\langle m\rangle}}

\newcommand{\lk}{\mathrm{link}(\sigma;\K)}

\newcommand{\Serre}{$(S_r)\ $}

\title{Sequentially $(S_r)$ is a topological property}

\author{Afshin Goodarzi}

\affil{\small{ Royal Institute of Technology, Department of Mathematics, S-100 44, Stockholm, Sweden.}}

\date{ }

\begin{document}

\maketitle

\begin{abstract}

We show that the sequentially \Serre condition for simplicial complexes is a topological property. Along the way, we present an elementary proof for the fact that the Serre's condition \Serre is a topological property. 

\end{abstract}

\section{The Serre's Condition \Serre}


  A $(d-1)$-dimensional simplicial complex $\K$ is said to satisfy the \emph{Serre's condition} \Serre over a field $\k$ if for each face $\sigma$ of $\K$ the $i^{\text{th}}$ reduced homology group $\wt{H}_i(\lk;\k)$ vanishes for all $i<\min\{r-1,d-\dim \sigma-2\}$. If $\K$ satisfies the Serre's condition ($S_r$), we simply say that $\K$ is ($S_r$). The concept of Serre's condition $(S_r)$ interpolates between the Cohen--Macaulayness and normality of complexes. On of extreme being ($S_d$) is equivalent to the Cohen--Macaulayness, while on the other extreme ($S_2$) is the familiar topological property of being normal. In a pioneering work \cite{MTC}, Munkres showed that Cohen--Macaulayness is a topological property, that is, it only depends on the geometric realization $|\K|$ of the complex $\K$ in question and possibly the characteristic of the field $\k$. It was shown by Yanagawa~\cite[Theorem 4.4]{Yanagawa2011} that the Serre's condition \Serre is a topological property.  Below, we give an alternative elementary proof of this fact that will be used in the next section.

Let $X$ be a polyhedron, i.e., a space homeomorphic to the geometric realization of some simplicial complex. Denote by $D_j(X;\k)$ the set of all points $p\in X$ such that the local homology group $H_j(X,X-p;\k)\neq 0$. Let $\Xm$ be the closure of the union of all $D_j(X;\k)$ for $j\geq m$. 

\begin{lem}
The set $\Xm$ is the closure of the set of all points $p\in X$ such that $p$ has a neighborhood in $X$ which is homeomorphic to an open $j$-dimensional ball for some $j\geq m$.   
\end{lem}

\begin{proof}
First notice that if $p$ has a neighborhood homeomorphic to an open $j$-ball for some $j\geq m$, then it follows from excision property that $H_j(X,X-p)\neq 0$, see \cite[Lemma 35.1]{Munk} for example. Thus, $p\in\Xm$.

Now, let $p\in D_j(X;\k)$ and let $\K$ be any triangulation of $X$. Let $\sigma$ be the unique face of $\K$ that contains $p$ in its interior. Also, let $\tau$ be a facet of $\K$ of maximum cardinality that contains $\sigma$. It follows from \cite[Lemma 3.3]{MTC} that \(\wt{H}_{j-\dim\sigma-1}(\lk;\k)\neq 0\). In particular, $j\leq \dim\tau$ because $j-\dim\sigma-1\leq \dim\lk=\dim\tau-\dim\sigma-1$. Finally, the conclusion follows as the interior of $\tau$ is an open ball of dimension $\geq j$ and $p\in\tau$.  
\end{proof}

Note that the result above justifies the absence of the field $\k$ in the notation of $\Xm$. The space $\Xm$ is a polyhedron. In fact, for a simplicial complex $\K$, if we let $\K^{\langle m\rangle}$ to be the subcomplex of $\K$ generated by all facets of dimension at least $m$, then $\Xm$ is the geometric realization of $\K^{\langle m\rangle}$, see \cite[Proposition 2.4]{BWW}. 

We let the \emph{dimension} of $D_k(X;\k)$ to be the maximum $\ell$ such that there is an $\ell$-dimensional open ball contained in $D_k(X;\k)$. Also, we use the convention that a set is empty if and only if its dimension is negative.  

\begin{thm}\label{STC}
Let $\K$ be a $(d-1)$-dimensional simplicial complex and $X=|\K|$ be its geometric realization. Then $\K$ is \Serre over $\k$ for some $2\leq r\leq d$ if and only if the following are satisfied:
\begin{enumerate}
\item[$(1)$] $X^{\langle d-1\rangle}=X$,
\item[$(2)$] $\wt{H}_j(X;\k)=0$ for all $j<r-1$, and 
\item[$(3)$] $\dim \left(D_k(X;\k)\right)\leq k-r$ for all $k\leq d-2$. 
\end{enumerate}
\end{thm}

\begin{proof}
First assume that $\K$ is \Serre over $\k$. The conditions $(1)$ and $(2)$ clearly hold. We verify $(3)$. Let $p\in D_k(X;\k)$ for some $k\leq d-2$ and $\sigma$ be the unique face of $\K$ such that $p\in\mathrm{int}(\sigma)$. We must show that $\dim \sigma\leq k-r$. We have 
\[ \wt{H}_{k-\dim \sigma-1}(\lk;\k)=H_k(X,X-p;\k)\neq 0.\]
Since $\K$ is \Serre over $\k$, we have $k-\dim \sigma-1>\min\{r-2,d-\dim\sigma -3\}$. Now, since $k\leq d-2$, it follows that $k-\dim \sigma-1>r-2$. Hence $\dim\sigma\leq k-r$ as desired. 

Next suppose that $\K$ is not \Serre over $\k$. We can assume $\K$ is pure, since otherwise the condition $(1)$ would be violated. There is a face $\sigma$ in $\K$ such that \( \wt{H}_{j}(\lk;\k)\neq 0\) for some $j\leq \min\{r-2,d-\dim\sigma -3\}$. If $\sigma$ is the empty-set, then the condition $(2)$ is not satisfied. For any point $p$ in the interior of $\sigma$ we have \(H_{j+\dim \sigma +1}(X,X-p;\k)\neq 0$. In particular, $\mathring{\sigma}\subseteq D_{j+\dim \sigma+1}$ which implies that 
\[\dim (D_{j+\dim \sigma+1})\geq \dim\sigma>\dim\sigma+1+j-r,\]
contradicting the condition $(3)$. 

\end{proof}

\begin{cor}[{{\cite[Theorem 4.4]{Yanagawa2011}}}]
The Serre's condition \Serre is a topological property. 
\end{cor}

 The property of being normal is, of course, independent from the field characteristic. The following example shows that being $(S_r)$ depends on the field characteristic for all $r\geq 3$.
\begin{ex}[{{cf. \cite[Question 2.5]{SerreSurvey}}}]
If $M$ is a $(d-1)$-dimensional manifold, then all the local homology groups $H_{\ast}(M,M-p;\k)$ are zero, except in the top-dimension, see \cite[Section 35]{Munk}. So, one can see, using Theorem \ref{STC} for instance, that $M$ is \Serre over $\k$ if and only if $\wt{H}_i(M;\k)=0$ for all $i<r-1$. In particular, the real projective space $\mathbb{R}\mathrm{P}^{d-1}$ is \Serre over a field of characteristic zero for all $r\geq 2$. However, if $\k$ has characteristic two, then $\wt{H}_2(\mathbb{R}\mathrm{P}^{d-1};\k)\neq 0$. Therefore, $\mathbb{R}\mathrm{P}^{d-1}$ is not \Serre over $\k$ if $r\geq 3$.
\end{ex}

\section{Sequentially \Serre}

Let $\K$ be a simplicial complex. The \emph{pure $m$-skeleton} $\K^{[m]}$ is the subcomplex generated by all $m$-dimensional faces. The pure $m$-skeleton can be also considered as the usual $m$-dimensional skeleton of the complex $\Km$. A simplicial complex $\K$ is \emph{sequentially \Serre}if all of its pure skeleta $\K^{[m]}$ are \Serre. Sequentially \Serre interpolates between the sequential Cohen--Macaulayness and the natural property of having connected face links. We refer to \cite{BWW,HTYZ,SerreSurvey} for more on these properties. Sequential Cohen--Macaulayness is a topological property, see \cite[Theorem 4.1.6]{W} for instance. It was asked in \cite[Question 4.14]{HTYZ} if the same hold true for all sequential \Serre conditions. We give an affirmative answer to this question. 

\begin{prop}\label{CS}
A $(d-1)$-dimensional simplicial complex $\K$ is sequentially \Serre over $\k$ if and only if for all $0\leq m\leq d-1$, all $\sigma\in\K^{\langle m\rangle}$, and all $j<\min\{m-\dim\sigma-1,r-1\}$ one has $\wt{H}_j(\mathrm{link}(\sigma;\Km);\k)=0$.
\end{prop}

\begin{proof}
Let $\sigma$ be a face of $\K$ and $m$ be a positive integer $\leq\dim\sigma$. It is not difficult to show that 
\begin{eqnarray*}
\mathrm{link}(\sigma;\Km)&=&\lk^{\langle m-\dim\sigma-1\rangle}\quad \text{and }\\ 
\mathrm{link}(\sigma;\K^{[m]})&=&\lk^{[ m-\dim\sigma-1]}.
\end{eqnarray*}
Note that $\lk^{[ m-\dim\sigma-1]}$ is the $(m-\dim\sigma-1)$-dimensional skeleton of $\lk^{\langle m-\dim\sigma-1\rangle}$ and, in particular, for all $j<m-\dim\sigma-1$ these two complexes must have the same homology groups.

\end{proof}

\begin{thm}\label{SSTC}
Let $\K$ be a $(d-1)$-dimensional simplicial complex and $X=|\K|$ be its geometric realization. Then $\K$ is sequentially \Serre over $\k$ for some $2\leq r\leq d$ if and only if the following are satisfied for all $0\leq m\leq d-1$:
\begin{enumerate}
\item[$(1)$] $\wt{H}_j(\Xm;\k)=0$ for all $j<\min\{m,r-1\}$, and 
\item[$(2)$] $\dim \left(D_k(\Xm;\k)\right)\leq k-r$ for all $k\leq m-1$. 
\end{enumerate}
\end{thm}
\begin{proof}

Let $\K$ be sequentially \Serre over $\k$. Fix $k$ and $m$ so that $k\leq m-1$. Let $p$ be a point in $D_k(\Xm;\k)$, and consider the unique $\sigma$ that $p\in\mathring{\sigma}$. Since $\wt{H}_{k-\dim\sigma-1}(\mathrm{link}(\sigma;\Km);\k)\neq 0$, it follows from Proposition \ref{CS} that $\dim\sigma\leq k-r$ as desired. 

Now, assume that $\K$ is not sequentially \Serre over $\k$. Then there exist an integer $m$, a face $\sigma$ of $\K$ and an integer $j<\min\{m-\dim\sigma-1,r-1\}$ such that $\wt{H}_j(\mathrm{link}(\sigma;\Km);\k)\neq 0$. Set $k:=j+\dim\sigma+1<m$. The rest of the argument is exactly as in the proof of Theorem \ref{STC}.

\end{proof}

\begin{cor}
Sequentially \Serre is a topological property.
\end{cor}

\vspace*{0.5cm}
\paragraph*{ Acknowledgement.}

This research has been supported by the grant KAW-stipendiet 2015.0360 from the Knut and Alice Wallenberg Foundation.

\def\cprime{$'$}

\end{document}